\documentclass[11pt,oneside,reqno]{amsart}
\usepackage[top=2.5cm,bottom=2.5cm,left=3cm,right=3cm]{geometry}

\usepackage{microtype}
\usepackage[T1]{fontenc}
\usepackage{lmodern}
\usepackage{amsmath,amssymb,amsthm,mathtools}
\usepackage{mathrsfs}

\DeclareFontFamily{U}{rsfs}{\skewchar\font127 }
\DeclareFontShape{U}{rsfs}{m}{n}{%
  <-6.5> rsfs5
  <6.5-8> rsfs7
  <8-> rsfs10
}{}
\usepackage[mathcal]{eucal}
\usepackage{enumitem}
\setlist{itemsep=1pt,topsep=4pt}
\numberwithin{equation}{section}
\usepackage[backend=biber,doi=false,url=false,isbn=false]{biblatex}
\usepackage{xcolor}
\definecolor{citeblue}{rgb}{0,0.35,0.75}
\usepackage[colorlinks=true,linkcolor=citeblue,citecolor=citeblue,urlcolor=citeblue]{hyperref}
\usepackage[nameinlink,noabbrev]{cleveref}
\Crefname{theorem}{Theorem}{Theorems}
\Crefname{conjecture}{Conjecture}{Conjectures}
\newtheorem{theorem}{Theorem}[section]
\newtheorem{proposition}[theorem]{Proposition}
\newtheorem{lemma}[theorem]{Lemma}
\newtheorem{corollary}[theorem]{Corollary}
\newtheorem{conjecture}[theorem]{Conjecture}
\theoremstyle{definition}
\newtheorem{definition}[theorem]{Definition}
\theoremstyle{remark}
\newtheorem{remark}[theorem]{Remark}

\DeclareMathOperator{\SO}{SO}
\DeclareMathOperator{\SL}{SL}
\DeclareMathOperator{\NS}{NS}
\DeclareMathOperator{\CH}{CH}
\DeclareMathOperator{\Aut}{Aut}

\newcommand{\Z}{\mathbb Z}
\newcommand{\Q}{\mathbb Q}
\newcommand{\R}{\mathbb R}
\newcommand{\OO}{\mathcal O}
\newcommand{\id}{\mathrm{id}}
\newcommand{\W}{\mathcal W}
\allowdisplaybreaks

\title{Bloch's conjecture for derived autoequivalences
of K3 surfaces}

\author{Zaiyuan Chen}
\address{Zaiyuan Chen \\ Shanghai Center for Mathematical Sciences \\ Fudan University \\ Jiangwan Campus \\ Shanghai \\ 200438 \\ China}
\email{zaiyuanchen16@fudan.edu.cn}

\author{Zhiyuan Li}
\address{Zhiyuan Li \\ Shanghai Center for Mathematical Sciences \\ Fudan University \\ Jiangwan Campus \\ Shanghai \\ 200438 \\ China}
\email{zhiyuan\_li@fudan.edu.cn}

\date{\today}

\begin{document}

\begin{abstract}
We prove Bloch's conjecture for (anti-)symplectic automorphisms
of smooth projective complex K3 surfaces of Picard number two.
Together with the recent result for abelian surfaces and the known
K3 cases, this settles the remaining case of Bloch's conjecture for
(anti-)symplectic automorphisms of smooth projective complex surfaces
with trivial canonical bundle.

More generally, we prove Bloch's conjecture for (anti-)symplectic
derived autoequivalences of projective K3 surfaces with Picard number
at least three, or with Picard number two and anisotropic
N\'eron--Severi lattice. The key new result is the congruence property
of the subgroup generated by reflective involutions in
$O(U\oplus N)$ for anisotropic even lattices $N$ of signature $(1,1)$.
Combined with the local generation theorem and the Chow-theoretic
criterion of Li--Yu--Zhang, this gives the required reflective
decompositions and their consequences for zero-cycles. For Picard
number one and for isotropic N\'eron--Severi lattices of rank two,
we further obtain modular-curve criteria for generation by reflective
involutions and unipotent isometries, including a complete
classification in the split case $N\cong U(d)$.
\end{abstract}
\maketitle

\section{Introduction}

Let $S$ be a smooth projective complex K3 surface. The
Bloch--Beilinson conjectures predict that the action of a
correspondence on zero-cycles is constrained by its action on
holomorphic forms. For automorphisms of $S$, this leads to the
following formulation of Bloch's conjecture.

\begin{conjecture}\label{conj:bloch}
For any $g_1,g_2\in\Aut(S)$,
\[
 g_1^*|_{T(S)}=g_2^*|_{T(S)}
 \quad\Longrightarrow\quad
 (g_1)_*=(g_2)_*\qquad\text{on }\CH_0(S),
\]
where $T(S)$ is the transcendental lattice of $S$.
\end{conjecture}

Equivalently, every symplectic automorphism $g$ of $S$ satisfies
\begin{equation}\label{eq:bloch-intro}
 g_*=\id\qquad\text{on }\CH_0(S).
\end{equation}
Here $g$ is symplectic if $g^*\omega_S=\omega_S$ for a nonzero
$\omega_S\in H^0(S,\Omega_S^2)$. For projective K3 surfaces,
this is equivalent to $g^*|_{T(S)}=\id$.
An automorphism is anti-symplectic if $g^*\omega_S=-\omega_S$,
equivalently $g^*|_{T(S)}=-\id$, and the corresponding prediction is
$g_*=-\id$ on $\CH_0(S)_{\mathrm{hom}}$.

\subsection{Summary of known results}

The following results establish the conjecture in several
important cases.

\begin{theorem}\label{thm:previous}
Let $g$ be a symplectic automorphism of $S$. Then $g_*=\id$ on
$\CH_0(S)$ in each of the following cases:
\begin{enumerate}[labelindent=0pt,leftmargin=*,labelsep=0.6em]
\item $\operatorname{rk}_2\NS(S)\geq4$ and
$\operatorname{rk}_3\NS(S)\geq3$ \cite{HuySurvey};
\item $g$ has finite order \cite{Voisin,HuyFinite};
\item $g$ preserves an elliptic fibration on $S$ \cite{DL};
\item $\rho(S)\geq3$ \cite[Theorem~1.3]{LYZ};
\item $\rho(S)=2$ and $\Aut(S)\not\cong\Z$
\cite[Section~7.2]{LYZ}.
\end{enumerate}
Here $\operatorname{rk}_p$ is the maximal rank of a sublattice
whose discriminant is prime to $p$.
\end{theorem}

We briefly recall the methods underlying these results.

\subsubsection*{1}
Huybrechts \cite{HuySurvey} relates the action of spherical twists on Chow
groups to reflections in Mukai cohomology. Kneser's generation
theorem then gives the conjecture under the local lattice
conditions in Theorem~\ref{thm:previous}(1).

\subsubsection*{2}
Voisin \cite{Voisin} proves the conjecture for symplectic involutions using
Prym varieties associated with curves on the quotient surface. Combining Voisin's theorem with his earlier results,
Huybrechts proves the conjecture for symplectic automorphisms
of arbitrary finite order \cite[Theorem~0.1]{HuyFinite}.

\subsubsection*{3}
Du--Liu \cite{DL} use the relative Jacobian and the group structure of
elliptic fibers to treat symplectic automorphisms preserving
an elliptic fibration. Their result also applies to
automorphisms of infinite order.

\subsubsection*{4}
Building on Huybrechts' approach, Li--Yu--Zhang study Bloch's
conjecture for derived autoequivalences. A Fourier--Mukai
autoequivalence $\Phi:D^b(S)\to D^b(S)$ induces an action
\[
 \Phi^{\CH}:\CH_0(S)_{\mathrm{hom}}
       \longrightarrow\CH_0(S)_{\mathrm{hom}}.
\]
We write $\Phi^H$ for its action on the Mukai lattice and
\[
 \Phi^{\mathrm{alg}}
 :=\Phi^H|_{\widetilde{\NS}(S)}
\]
for its algebraic Mukai action.
Here $\CH_0(S)_{\mathrm{hom}}$ is the degree-zero subgroup.
The autoequivalence is symplectic (resp.\ anti-symplectic) if its
cohomological action on $H^{2,0}(S)$ is $\id$ (resp.\ $-\id$),
equivalently if its restriction to $T(S)$ is $\id$
(resp.\ $-\id$). Li--Yu--Zhang relate its action on
zero-cycles to the induced isometry of the algebraic Mukai
lattice
\[
 \widetilde{\NS}(S)=\Z\oplus\NS(S)\oplus\Z
 \simeq U\oplus\NS(S).
\]
Their main criterion is the following.

\begin{theorem}[{\cite[Theorem~1.2]{LYZ}}]
\label{thm:reflective-criterion}
Let $\Phi$ be a symplectic (resp.\ anti-symplectic) derived
autoequivalence of $S$.
If its action on $\widetilde{\NS}(S)$ is a product of
reflective involutions, then
\[
 \Phi^{\CH}=\id\quad\text{(resp.\ $-\id$)}\qquad\text{on }\CH_0(S)_{\mathrm{hom}}.
\]
\end{theorem}

\begin{remark}
For $\rho(S)\leq2$, one may also allow unipotent isometries acting
trivially on the discriminant group among the factors in
Theorem~\ref{thm:reflective-criterion}. The conclusion on
$\CH_0(S)_{\mathrm{hom}}$ is the same, with the sign determined
by the action of $\Phi$ on $H^{2,0}(S)$.
The unipotent case is Proposition~\ref{prop:unipotent-chow};
the assertion for products follows from
\cite[Theorem~3.6]{LYZ}.
\end{remark}

The connection with the stable orthogonal group is as follows.
By oriented derived Torelli \cite{HMSOrientation}, every derived
autoequivalence preserves the orientation of the positive two-planes
in $\widetilde{\NS}(S)$. Since the Mukai lattice is unimodular,
the actions on the discriminant groups of
$\widetilde{\NS}(S)$ and its orthogonal complement $T(S)$
agree under the natural identification. Consequently,
\begin{equation}\label{eq:mukai-stable}
 \begin{aligned}
 \Phi\text{ symplectic}
 &\quad\Longrightarrow\quad
 \Phi^{\mathrm{alg}}\in
 \widetilde O^+\bigl(\widetilde{\NS}(S)\bigr),\\
 \Phi\text{ anti-symplectic}
 &\quad\Longrightarrow\quad
 -\Phi^{\mathrm{alg}}\in
 \widetilde O^+\bigl(\widetilde{\NS}(S)\bigr).
 \end{aligned}
\end{equation}
Here $\widetilde O^+$ consists of orientation-preserving
isometries acting trivially on the discriminant group.
For the second implication, $-\id$ preserves the orientation
of positive two-planes. Equivalently, $\Phi[1]$ is symplectic,
and $(\Phi[1])^{\CH}=-\Phi^{\CH}$. Thus the anti-symplectic
case reduces to the symplectic case by a shift.

For $\rho(S)\geq3$, Li--Yu--Zhang \cite[Section~7.2]{LYZ} obtain the required
decomposition by proving that the reflective subgroup is a
congruence subgroup; see Theorem~\ref{thm:congruence-reflective}.
For $\rho(S)=2$, they prove Conjecture~\ref{conj:bloch} when
$\Aut(S)\not\cong\Z$. Thus the only
automorphism case not covered by Theorem~\ref{thm:previous} has
$\rho(S)=2$ and $\Aut(S)\cong\Z$, in which case $\NS(S)$ is
anisotropic \cite[Lemma~7.6]{LYZ}.

\subsection{Main results}
We extend the reflective decomposition method of \cite{LYZ} to
projective K3 surfaces of Picard numbers one and two. As we will see in
\S~\ref{subsec:modular}, one can associate a smooth compact
modular curve $X_0$ with each of the lattices
\[
 \langle2d\rangle
 \qquad\text{and}\qquad
 \begin{pmatrix}0&d\\d&2a\end{pmatrix},
 \qquad d>0,\quad a\in\Z.
\]
We denote the genus of $X_0$ by $\mathbf{gen}(d)$ and
$\mathbf{gen}(d,a)$, respectively. The main result is the following.

\begin{theorem}\label{thm:derived}
Let $S$ be a smooth projective complex K3 surface. Assume that
one of the following conditions holds:
\begin{enumerate}[label=\textup{(\roman*)},leftmargin=*]
\item $\rho(S)\geq3$.
\item $\rho(S)=2$ and $\NS(S)$ is anisotropic. 
\item $\NS(S)\cong  \langle 2d\rangle$, and $\mathbf{gen}(d)=0$.
\item $\NS(S)\cong
\left(\begin{smallmatrix}0&d\\d&2a\end{smallmatrix}\right)$
for some $d>0$ and $a\in\Z$, and $\mathbf{gen}(d,a)=0$.
\end{enumerate}
Then every symplectic
(resp.\ anti-symplectic) derived autoequivalence $\Phi$ of $S$ satisfies
\[
 \Phi^{\CH}=\id\quad\text{(resp.\ $-\id$)}
 \qquad\text{on }\CH_0(S)_{\mathrm{hom}}.
\]
The assertion also holds with integral coefficients.
\end{theorem}

Together with the known cases, the anisotropic rank-two result
settles the remaining case for automorphisms.

\begin{theorem}\label{thm:main}
Conjecture~\ref{conj:bloch} holds for every smooth projective
complex K3 surface. Moreover, if $g$ is a symplectic
(resp.\ anti-symplectic) automorphism of $S$, then
\[
 g_*=\id\quad\text{(resp.\ $-\id$)}
 \qquad\text{on }\CH_0(S)_{\mathrm{hom}}.
\]
\end{theorem}

\begin{remark}
The lattice-theoretic criteria in Theorem~\ref{thm:derived} are
optimal for our current approach. See Proposition~\ref{prop:split-generation} and
Corollary~\ref{cor:rank-one-generation} for the precise statements.
\end{remark}

\begin{remark}
    For abelian surfaces, Chen--Li--Zhang \cite[Theorem~A]{CLZAbelian}
prove the corresponding statement for all (anti-)symplectic
derived autoequivalences on the Albanese kernel.
Together with Theorem~\ref{thm:main}, this establishes Bloch's
conjecture for (anti-)symplectic automorphisms of every smooth
projective complex surface with trivial canonical bundle.
\end{remark}

\subsection{Congruence subgroup problem}

The proof of Theorem~\ref{thm:derived} reduces to a
generation problem for the stable orthogonal group of
the algebraic Mukai lattice. In the anisotropic case,
the main input is the following theorem.

\begin{theorem}[=Proposition \ref{prop:binary-congruence}]\label{thm:anisotropic-congruence}
Let $N$ be an even lattice of signature $(1,1)$, and let $\W(U\oplus N)$ be the subgroup generated by reflective involutions.
If $N$ contains no nonzero isotropic vector, then
$\W(U\oplus N)$ is a congruence subgroup of $O(U\oplus N)$. As a consequence,
$\W(U\oplus N)=\widetilde O(U\oplus N)$.
\end{theorem}

This result follows the argument in 
\cite[Corollary~6.11]{LYZ}.  The anisotropic assumption identifies the rational
spin group with $\operatorname{Res}_{F/\Q}\SL_2$ for the
real quadratic field $F=\Q(\sqrt{-\det N})$. Margulis's normal subgroup theorem
and Serre's theorem (Theorems~\ref{thm:margulis} and \ref{thm:serre}) then imply the congruence property. The last assertion is a consequence of the results in  \cite{LYZ} (cf.~Theorem \ref{thm:congruence-reflective}), and hence Theorem~\ref{thm:derived} \textup{(ii)} follows.

For the remaining cases, reflective involutions alone
need not generate the stable orthogonal group.
We therefore also use unipotent isometries.
Let $K(\widetilde{\NS}(S))$ be the subgroup of
$\widetilde O^+(\widetilde{\NS}(S))$ generated by
orientation-preserving products of reflective
involutions and by unipotent isometries.

\begin{theorem}\label{thm:low-quotients}
With the notation above,
\[
 K(\widetilde{\NS}(S))
 =
 \widetilde O^+(\widetilde{\NS}(S))
\]
if $X_0$ has genus zero.
\end{theorem}

Oriented derived Torelli realizes every unipotent isometry in
$\widetilde O^+(\widetilde{\NS}(S))$ as the algebraic Mukai action of
a symplectic derived autoequivalence. Its Chow action is controlled by
the following result.

\begin{theorem}[=Proposition~\ref{prop:unipotent-chow}]
\label{thm:unipotent-chow}
Let $S$ be a smooth projective complex K3 surface with
$\rho(S)\leq2$, and let $\Phi$ be a symplectic derived
autoequivalence of $S$. If $\Phi^{\mathrm{alg}}$ is unipotent, then
\[
 \Phi^{\CH}=\id
 \qquad\text{on }\CH_0(S)_{\mathrm{hom}}.
\]
\end{theorem}

Together with Theorem~\ref{thm:reflective-criterion},
these two results give
Theorem~\ref{thm:derived} \textup{(iii)} and \textup{(iv)}.

\subsection*{Organization of the paper}
Section~\ref{sec:lattices} introduces reflective and unipotent
isometries, the relevant congruence criterion, and the associated
modular curves.

Section~\ref{sec:picard-two} treats Picard number two, proving the
anisotropic congruence theorem and the generation results in the
isotropic case, including the exact criterion for $U(d)$.

Section~\ref{sec:picard-one} computes
$\widetilde O^+(L)/K(L)$ for $L=U\oplus\langle2d\rangle$ and completes
the proofs of Theorems~\ref{thm:low-quotients} and
\ref{thm:derived}.

\subsection*{Acknowledgements}

The authors thank Ruxuan Zhang for helpful discussions, 
and Zheng Yuan and Zhihong Wang for their kind help.
This work was supported by NSFC grant 12425105 and the
Shanghai Pilot Program for Basic Research (No.~21TQ001).

\section*{History and AI disclosure}

Our approach to Bloch's conjecture was to decompose complicated
derived autoequivalences into simpler ones whose actions on
zero-cycles are understood. However, carrying out this
decomposition directly in the derived category of a K3 surface
runs into difficulties at low Picard number. To address these difficulties,
we pursued two approaches over the past two years:
\begin{itemize}
\item[A)] Lift the problem to a higher-dimensional hyper-K\"ahler
variety, for instance a variety of \(K3^{[n]}\)-type or
O'Grady's ten-dimensional exceptional hyper-K\"ahler type.
The advantage of this approach is that the increased rank
of the relevant lattice ensures that the desired decomposition
exists. The difficulty is that the corresponding instances
of Bloch's conjecture are not known for all the reflections
occurring in such a decomposition.

\item[B)] Restrict the decomposition problem to the actions
induced by automorphisms, rather than arbitrary autoequivalences.
These actions are represented by certain diagonal matrices,
so the aim is to decompose only these matrices into reflective
involutions. The difficulty is that we also knew of
counterexamples to the desired decomposition for diagonal
matrices in general.
\end{itemize}

The turning point came from an experiment rather than a planned
step in the proof. While testing Seed Pro 2.1, the second author asked whether
the reflection group and the diagonal group could have a
nontrivial intersection. The model first proposed a
counterexample already known to us. He then asked whether it could find one under
the additional assumption that the N\'eron--Severi lattice
contains no nonzero vector of square \(0\), \(2\), or \(-2\).
For a projective K3 surface of Picard number two, this condition
is equivalent to \(\Aut(S)\cong\Z\). In response, it drew
his attention to Serre's congruence subgroup theorem.

That suggestion became the starting point of this paper.
Building on it, one can easily use our previous approach in \cite{LYZ} to prove that, when the N\'eron--Severi lattice
has rank two and is anisotropic, the associated reflective group
is a congruence subgroup.
We record this exchange both to acknowledge that contribution
and to preserve an unexpected part of the history of the work.
LLMs were also used to refine and improve
the exposition of the manuscript.

\noindent \textit{P.S.} With further input from the first author and assistance from LLMs, we strengthen our results to cover the remaining cases for symplectic automorphisms in Picard number one and two.

\section{Orthogonal transformations of lattices}\label{sec:lattices}

\subsection{Reflective involutions}

Let $U=\Z e\oplus\Z f$ be the hyperbolic plane, with $e^2=f^2=0$ and $(e,f)=1$.
For an even nondegenerate lattice $L$, put $q_L(x)=(x,x)/2$, and write $L^\vee$ for its dual lattice
and $A_L=L^\vee/L$ for its discriminant group.
The stable orthogonal group is defined to be
\[
 \widetilde O(L)
 =
 \ker\bigl(O(L)\longrightarrow O(A_L)\bigr).
\]
If $L$ has signature $(2,\nu)$, let $O^+(L)$ be the
subgroup preserving an orientation of the positive
two-planes, and set
\[
 \widetilde O^+(L)=\widetilde O(L)\cap O^+(L).
\]
We recall the reflective involutions of \cite[Definition~3.3]{LYZ}.

\begin{definition}\label{def:reflective}
A reflective involution of $L$ is an integral isometry
of one of the following forms:
\begin{enumerate}[label=\textup{(\roman*)},leftmargin=*]
\item a reflection
\[
 s_r(x)=x-\frac{2(x,r)}{(r,r)}r,
 \qquad r^2=\pm2;
\]
\item an involution $j$ acting trivially on $A_L$
for which there is a primitive isotropic vector
$u\in L$ such that
\[
 j(u)=-u,
 \qquad
 j|_{u^\perp/\Z u}=\id.
\]
\end{enumerate}
We denote by $\W(L)$ the subgroup generated by these
involutions.
\end{definition}

Since the defining conditions are preserved by conjugation, 
$\W(L)$ is a normal subgroup of $O(L)$ contained in $\widetilde O(L)$.
For $L=U\oplus N$ and $v\in N$, $\W(L)$ contains the integral Eichler transformations $T_e(v)$ and $T_f(v)$; see \cite[Section~6.4]{LYZ}. In this paper we use the convention
\[
T_e(v)e=e,\qquad T_e(v)f=f+v-\frac{v^2}{2}e, \qquad  T_e(v)x=x-(v,x)e
 \ (x\in N),
\]
and define $T_f(v)$ by interchanging $e$ and $f$.

\subsection{Congruence subgroups}

For an integral lattice $M$, a subgroup of $O(M)$ is a
\emph{congruence subgroup} if it contains
the kernel of
\[
 O(M)\longrightarrow\Aut(M/mM).
\]
for some integer $m\geq1$.

\begin{theorem}
\label{thm:congruence-reflective}
Let $N$ be an even lattice of signature $(1,\rho-1)$
with $\rho\geq2$, and put $L=U\oplus N$.
Suppose $B\leq O(L)$ contains $\W(L)\cap O^+(L)$ and $B\cap\SO(L)$ is a congruence subgroup of
$\SO(L)$. Then $\widetilde O^+(L)\subseteq B$. In particular, if
$\W(L)$ is a congruence subgroup of $O(L)$, then
\begin{equation}\label{eq:reflection-eq}
 \W(L)=\widetilde O(L).
\end{equation}
\end{theorem}

\begin{proof}
Let $\gamma\in\widetilde O^+(L)$. Multiplying by $s_{e-f}\in B$ if
necessary, assume $\det\gamma=1$. On $L(-1)$ the rational spinor norm
of $\gamma$ is trivial by \cite[Lemma~6.1]{LYZ}. Since
$B\cap\SO(L)$ is congruence, the local decomposition and simultaneous
approximation in
\cite[Theorem~6.9(3$'$) and the proof of (1$'$)]{LYZ}
give $w\in\W(L)\cap O^+(L)$ such that
$w^{-1}\gamma\in B\cap\SO(L)$. Hence $\gamma\in B$.

Take $B=\W(L)$ for the last assertion. The reflection $s_{e+f}$
reverses the orientation of positive two-planes, so $\W(L)$ also
contains the other coset of $\widetilde O^+(L)$ in $\widetilde O(L)$.
Conversely, all generators of $\W(L)$ act trivially on $A_L$.
\end{proof}

By \cite[Corollary~6.11]{LYZ}, the congruence hypothesis holds
when $\rho\geq3$.
Although \eqref{eq:reflection-eq} may fail when $\rho=2$, we will establish the congruence hypothesis directly for anisotropic $N$ of rank two in 
Proposition~\ref{prop:binary-congruence}.
We recall the two arithmetic results used in its proof.

\begin{theorem}[Margulis, \cite{Margulis}]\label{thm:margulis}
Let $\Gamma$ be an irreducible lattice in
$\SL_2(\R)\times\SL_2(\R)$.
Every infinite normal subgroup of $\Gamma$ has
finite index.
\end{theorem}

\begin{theorem}[{Serre, \cite[p.~499, Corollary~3]{SerreCSP}}]\label{thm:serre}
Let $F$ be a real quadratic field with ring of
integers $\OO_F$.
Every finite-index subgroup of $\SL_2(\OO_F)$ contains
\[
 \ker\bigl(\SL_2(\OO_F)\longrightarrow
 \SL_2(\OO_F/\ell\OO_F)\bigr)
\]
for some integer $\ell>0$.
\end{theorem}

\subsection{Unipotent isometries}

An isometry $\gamma\in O(L)$ is unipotent if $\gamma-\id$ is
nilpotent on $L_\Q$.
For a lattice of signature $(2,\nu)$, define
\begin{equation}\label{eq:low-K}
 K(L)=
 \left\langle
 \W(L)\cap O^+(L),\
 \gamma\in\widetilde O^+(L)\text{ unipotent}
 \right\rangle.
\end{equation}
Both sets of generators are preserved by conjugation,
so $K(L)$ is a normal subgroup of
$\widetilde O^+(L)$.

The following elementary observation will be used
to study the Chow action of unipotent isometries.

\begin{lemma}\label{lem:unipotent-isotropic}
Let $L$ be a lattice of signature $(2,\rho)$,
where $1\leq\rho\leq2$.
Every nonidentity unipotent isometry $\gamma\in O(L)$
fixes a primitive isotropic vector $u\in L$ and
acts trivially on $u^\perp/\Z u$.
In particular, there is a vector $w\in u^\perp$
such that
\[
 w^2>0,
 \qquad
 \gamma w-w\in\Z u.
\]
\end{lemma}

\begin{proof}
The endomorphism $A=\log\gamma$ is nonzero, nilpotent,
and rational.
Since $\gamma$ is an isometry,
\[
 A^*=\log(\gamma^*)=\log(\gamma^{-1})=-A.
\]
Choose $r\geq2$ such that $A^r=0$ and
$A^{r-1}\neq0$, and choose $v\in L_\Q$ with
$u=A^{r-1}v\neq0$.
Then $Au=0$ and
\[
 (u,u)=(-1)^{r-1}(v,A^{2r-2}v)=0.
\]
Rescaling $u$, we may assume that it is primitive
in $L$.

The induced isometry on $u^\perp_\Q/\Q u$ is
unipotent.
This quadratic space has signature $(1,\rho-1)$
and dimension at most two, so it has no nonidentity
unipotent isometries.
Thus $\gamma$ acts trivially on this quotient.
Primitivity of $u$ gives
\[
 \Q u\cap L=\Z u,
\]
hence $\gamma$ also acts trivially on $u^\perp/\Z u$.

The quotient $u^\perp_\Q/\Q u$ contains a rational
vector of positive square.
Lifting it to $u^\perp_\Q$ and clearing denominators
gives the required vector $w$.
\end{proof}

Let us explain why unipotent isometries behave well in Bloch's conjecture. The following result is a direct consequence of \cite[Theorem~3.9]{LYZ}.

\begin{proposition}\label{prop:unipotent-chow}
Let $S$ be a smooth projective complex K3 surface with
$\rho(S)\leq2$, and let $\Phi$ be a symplectic derived
autoequivalence of $S$. If $\Phi^{\mathrm{alg}}$ is unipotent, then
\[
 \Phi^{\CH}=\id
 \qquad\text{on }\CH_0(S)_{\mathrm{hom}}.
\]
\end{proposition}

 \begin{proof}
Put $\gamma=\Phi^{\mathrm{alg}}$. If $\gamma=\id$, the assertion follows
from \cite[Theorem~3.6]{LYZ}. Otherwise,
 Lemma~\ref{lem:unipotent-isotropic} gives a primitive isotropic vector
 $u$ and a positive class in $u^\perp/\Z u$ fixed by $\gamma$.
By \cite[Theorem~3.9]{LYZ}, the assertion follows.
\end{proof}

\begin{remark}
The restriction $\rho(S)\leq2$ is essential to this argument.
If $\rho(S)\geq3$, then $u^\perp_\Q/\Q u$ has signature
$(1,\rho(S)-1)$ and may itself admit nonidentity unipotent
isometries. Thus a unipotent algebraic Mukai action need not act
trivially on this quotient or fix a positive class there, so \cite[Theorem~3.9]{LYZ} does not apply in general.
\end{remark}

\subsection{Associated modular curves}\label{subsec:modular}

To an even lattice
\[
 N=\langle2d\rangle
 \qquad\text{or}\qquad
 N=\begin{pmatrix}0&d\\d&2a\end{pmatrix},
\]
where $d>0$ and $a\in\Z$, we associate a smooth
compact modular curve as follows.

For a positive integer $\ell$, recall that
\[
 \begin{aligned}
 \Gamma_0(\ell)
 &=
 \left\{
 \begin{pmatrix}\alpha&\beta\\ \gamma&\delta\end{pmatrix}
 \in\SL_2(\Z):
 \gamma\equiv0\pmod\ell
 \right\},\\
 \Gamma_1(\ell)
 &=
 \left\{
 \begin{pmatrix}\alpha&\beta\\ \gamma&\delta\end{pmatrix}
 \in\Gamma_0(\ell):
 \alpha\equiv\delta\equiv1\pmod\ell
 \right\}.
 \end{aligned}
\]
We also write
\[
 w_\ell=
 \begin{pmatrix}
 0&-1/\sqrt\ell\\
 \sqrt\ell&0
 \end{pmatrix}
\]
for the Fricke involution of level $\ell$.

\begin{definition}
For $N=\langle2d\rangle$, let $X_0$ be the smooth
compactification of
\[
 \langle\Gamma_0(d),w_d\rangle\backslash\mathbb H.
\]
For
$N=\left(\begin{smallmatrix}0&d\\d&2a\end{smallmatrix}\right)$,
let $X_0$ be the smooth
compactification of
\[
 \left(
 \Gamma_0\!\left(\frac{d^2}{\gcd(d,a)}\right)
 \cap\Gamma_1(d)
 \right)\backslash\mathbb H.
\]
Here $\gcd(d,0)=d$.
In the two cases, we denote the genus of $X_0$ by
$\mathbf{gen}(d)$ and $\mathbf{gen}(d,a)$, respectively.
\end{definition}

We recall the cusp data for $\Gamma_0(\ell)$ that will be used below
\cite[Section~3.8]{DiamondShurman}.
Its cusps are represented by reduced fractions $b/c$, where
\[
 c\mid\ell,
 \qquad
 b\in\bigl(\Z/\gcd(c,\ell/c)\Z\bigr)^\times,
\]
and $b$ may be chosen coprime to $\ell$.
Since $\kappa(b/c)=-b/c$ and $w_\ell(b/c)=-1/((\ell/c)b)$, complex
conjugation $\kappa(\tau)=-\bar\tau$ and the Fricke involution act
on the cusp labels $(c,b)$ by
\begin{equation}\label{eq:p1-cusp-action}
 \kappa(c,b)=(c,-b),
 \qquad
 w_\ell(c,b)=(\ell/c,-b).
\end{equation}
In particular, if $c^2=\ell$, then $w_\ell(c,b)=\kappa(c,b)$, so the
corresponding cusp has real image in the Fricke quotient.

For a reduced cusp $b/c$ of $\Gamma_0(\ell)$, its width is
\[
 v_{b/c}^{(\ell)}=\frac{\ell}{\gcd(c^2,\ell)},
\]
and conjugating
$\left(\begin{smallmatrix}1&v_{b/c}^{(\ell)}\\0&1\end{smallmatrix}\right)$
by a matrix in $\SL_2(\Z)$ taking $\infty$ to $b/c$ gives the following
generator of the cusp stabilizer modulo $\{\pm I_2\}$:
\begin{equation}\label{eq:p2-parabolic}
 p_{b/c}^{(\ell)}=
 \begin{pmatrix}
 1-bcv_{b/c}^{(\ell)}&b^2v_{b/c}^{(\ell)}\\
 -c^2v_{b/c}^{(\ell)}&1+bcv_{b/c}^{(\ell)}
 \end{pmatrix}\in\Gamma_0(\ell).
\end{equation}

The standard formula in
\cite[Section~3.1]{DiamondShurman} computes $\mathbf{gen}(d,a)$
and the genus of the classical modular curve of level $d$.
For the curve $X_0$ associated with $N=\langle2d\rangle$, we have the
following class-number formula.

\begin{proposition}\label{prop:p1-genus}
One has $\mathbf{gen}(d)=0$ for $d\leq4$.
For $d\geq5$, 
\[
 \mathbf{gen}(d)=
 \frac{\operatorname{genus}\!\left(
 \overline{\Gamma_0(d)\backslash\mathbb H}\right)+1}{2}-\frac14
 \begin{cases}
 \mathbf{cl}(-4d)+\mathbf{cl}(-d),&d\equiv3\pmod4,\\
 \mathbf{cl}(-4d),&d\not\equiv3\pmod4,
 \end{cases}
\]
where $\mathbf{cl}(D)$ is the class number of primitive positive
definite binary quadratic forms of discriminant $D<0$.
\end{proposition}

\begin{proof}
For $1\leq d\leq4$, the classical modular curve of level $d$ has
genus zero, hence so does $X_0$. For $d\geq5$, Riemann--Hurwitz for
the degree-two Fricke quotient map
\[
 \overline{\Gamma_0(d)\backslash\mathbb H}\longrightarrow X_0
\]
gives
\[
2\operatorname{genus}\!\left(
\overline{\Gamma_0(d)\backslash\mathbb H}\right)-2
=2\bigl(2\mathbf{gen}(d)-2\bigr)+\#\operatorname{Fix}(w_d).
\]
Substitution of the Fricke fixed-point count on the classical curve
\cite[pp.~453--454]{OggHyperelliptic} gives the formula.
\end{proof}

\section{K3 surfaces of Picard number two}
\label{sec:picard-two}

\subsection{Anisotropic case}

Let $N$ be an even lattice of signature $(1,1)$ and put
$L=U\oplus N$.

\begin{proposition}\label{prop:binary-congruence}
If $N$ is anisotropic, then $\W(L)=\widetilde O(L)$.
In particular, $\W(L)$ is a congruence subgroup of $O(L)$.
\end{proposition}

\begin{proof}
By Theorem~\ref{thm:congruence-reflective}, it suffices to prove the
congruence assertion. Put $F=\Q(\sqrt{-\det N})$. Since $N$ is
anisotropic, $F$ is real quadratic and
$\operatorname{Spin}(L_\Q)\simeq\operatorname{Res}_{F/\Q}\SL_2$.
The inverse image $\Gamma$ of $\SO(L)$ is an irreducible
arithmetic lattice in $\SL_2(\R)^2$, commensurable with
$\SL_2(\OO_F)$. The inverse image of $\W(L)\cap\SO(L)$ in $\Gamma$
is an infinite normal subgroup by \cite[Satz~4]{Kn81}
and its extension in \cite[Section~6.3]{LYZ}.
Theorems~\ref{thm:margulis} and~\ref{thm:serre} therefore show that
this inverse image is a congruence subgroup. 

The standard descent
through the spin covering
(cf. \cite[Satz~1]{Kn81} and \cite[Theorem~6.2(1)]{LYZ}) shows that $\W(L)$ is a
congruence subgroup of $O(L)$. Theorem~\ref{thm:congruence-reflective}
now yields $\W(L)=\widetilde O(L)$.
\end{proof}

\subsection{Proof of
Theorem~\ref{thm:derived} \textup{(ii)} and Theorem \ref{thm:main}}

For Theorem~\ref{thm:derived}\textup{(ii)}, take $N=\NS(S)$ and
$L=\widetilde{\NS}(S)$. Proposition~\ref{prop:binary-congruence}
gives $\W(L)=\widetilde O(L)$, and hence
$\W(L)\cap O^+(L)=\widetilde O^+(L)$.
Theorem~\ref{thm:reflective-criterion} gives the
assertion for symplectic autoequivalences.
The anti-symplectic case follows by a shift.

For Theorem \ref{thm:main}, consider the Picard number of $S$:
\begin{itemize}
    \item For $\rho(S)\geq3$, the result follows from
\cite{LYZ}.
    \item For $\rho(S)=2$ with $\Aut(S)\not\cong\Z$,
it follows from \cite[Section~7.2]{LYZ}.
In the remaining rank-two case,
$\Aut(S)\cong\Z$, and $\NS(S)$ is anisotropic
by \cite[Lemma~7.6]{LYZ}.
Theorem~\ref{thm:derived} \textup{(ii)} then applies.
\item For $\rho(S)=1$,  $\Aut(S)$ is either trivial or $\Z/2\Z$. The assertion is immediate. \qed
\end{itemize}

\subsection{Isotropic case}
\label{subsec:general-isotropic} Let us deal with the case where
$\NS(S)$ is isotropic. Then
$\NS(S)\cong
\left(\begin{smallmatrix}0&d\\d&2a\end{smallmatrix}\right)$
for some $d>0$ and $a\in\Z$.

\begin{proposition}\label{prop:general-isotropic}
Let
$N=\left(\begin{smallmatrix}0&d\\d&2a\end{smallmatrix}\right)$,
where $d>0$ and $0\leq a<d$, and put $L=U\oplus N$.
If $\mathbf{gen}(d,a)=0$, then
\[
 K(L)=\widetilde O^+(L).
\]
\end{proposition}

\begin{proof}
Put $m=d/\gcd(d,a)$. We identify $L$ with the lattice
\[
 \left\{
 \begin{pmatrix}r&dx+ay\\-y&s\end{pmatrix}
 :r,s,x,y\in\Z
 \right\}
 \subseteq M_2(\Q),
\]
equipped with the quadratic form $q_L(X)=\det(X)$.
Consider the congruence subgroup
\[
 \Gamma=
 \left\{
 \begin{pmatrix}\alpha&\beta\\\gamma&\delta\end{pmatrix}
 \in\SL_2(\Z):
 \alpha\equiv\delta\equiv1\pmod d,
 \beta\equiv0\pmod d,
 \gamma\equiv0\pmod m
 \right\}.
\]

A direct calculation shows that, for every $P\in\Gamma$,
the isometries
\[
 X\longmapsto PX,\qquad X\longmapsto XP^{-1}
\]
preserve $L$ and induce the identity on $A_L$.
They also preserve the orientation of positive two-planes,
since they belong to the image of
$\SL_2(\R)\times\SL_2(\R)$.
Moreover,
\[
 \operatorname{diag}(d,1)^{-1}\Gamma
 \operatorname{diag}(d,1)
 =\Gamma_0(dm)\cap\Gamma_1(d),
\]
so the compact modular curve associated with $\Gamma$
has genus $\mathbf{gen}(d,a)$.

For $d\geq5$, the group $\Gamma$ is torsion-free. Since its
compact modular curve has genus zero, the usual presentation shows
that $\Gamma$ is generated by cusp stabilizers. Each such stabilizer
is generated by a matrix of trace $\pm2$. Since
$\operatorname{tr}P\equiv2\pmod d$ for every $P\in\Gamma$, these
matrices have trace $2$ and are therefore unipotent.

For $d\leq4$, the group $\Gamma$ contains $\Gamma(3)$
when $d=3$ and $\Gamma(4)$ otherwise.
By \cite[Chapter~3]{DiamondShurman}, both principal congruence
subgroups are generated by unipotent matrices.
Thus, in every case, $\Gamma$ contains a principal
congruence subgroup whose left and right actions
belong to $K(L)$.
Under the
spin covering
\[
 \SL_2(\Q)\times\SL_2(\Q)\longrightarrow\SO(L_\Q),
 \qquad
 (P,Q)\longmapsto\bigl(X\mapsto PXQ^{-1}\bigr),
\]
the inverse image of $K(L)\cap\SO(L)$ therefore contains
$\Gamma(\ell)\times\Gamma(\ell)$ for some $\ell>0$.
The standard descent through the spin covering shows that
$K(L)\cap\SO(L)$ is a congruence subgroup of $\SO(L)$.
Theorem~\ref{thm:congruence-reflective}, applied with $B=K(L)$,
now gives $K(L)=\widetilde O^+(L)$.
\end{proof}

\subsection{Proof of
Theorem~\ref{thm:derived} \textup{(iv)}}
Apply Proposition~\ref{prop:general-isotropic}
to $L=\widetilde{\NS}(S)$.
The equality
\[
 K(\widetilde{\NS}(S))
 =
 \widetilde O^+(\widetilde{\NS}(S))
\]
and Theorem~\ref{thm:reflective-criterion},
Proposition~\ref{prop:unipotent-chow}, and
\cite[Theorem~3.6]{LYZ} give the assertion for symplectic
autoequivalences.
The anti-symplectic case follows by a shift.\qed 

\subsection{A complete classification and counterexamples}
\label{subsec:split-generation}

The condition $\mathbf{gen}(d,a)=0$ in
Proposition~\ref{prop:general-isotropic} is not necessary. To see this, we give an explicit criterion for the case $\NS(S)\cong U(d)$.

\begin{proposition}\label{prop:split-generation}
Let $L=U\oplus U(d)$. Then
\[
 K(L)=\widetilde O^+(L)
 \quad\Longleftrightarrow\quad
 d\in\{1,2,3,4,5,6,7,8,9,10,12,13,16,18,25\}.
\]
Moreover,
\[
 \W(L)\cap O^+(L)=\widetilde O^+(L)
 \quad\Longleftrightarrow\quad
 d\in\{1,2,3,4,5,6,7,8,10,12,13\}.
\]
\end{proposition}

\begin{proof}
Realize $L$ as the lattice of matrices
$\left(\begin{smallmatrix}x&z\\-dw&y\end{smallmatrix}\right)$
with quadratic form $\det$. Put
\[
 G=\bigl(\Gamma_0(d)/\{\pm I_2\}\bigr)\rtimes\langle D\rangle,
 \qquad D=\operatorname{diag}(1,-1),
\]
and define
\[
 \theta:G^{\mathrm{ab}}\longrightarrow
 \bigl((\Z/d\Z)^\times/\{\pm1\}\bigr)\times\{\pm1\},
 \qquad [P]\longmapsto\bigl([P_{11}],\det P\bigr).
\]
For $P,Q\in G$ of common determinant $\varepsilon$, choose matrix
representatives and let $X\mapsto PXQ^{-1}$. On the standard
generators of $A_L\simeq(\Z/d\Z)^2$, this map acts by the scalars
$\varepsilon P_{11}Q_{11}$ and $\varepsilon P_{22}Q_{22}$.
These scalars are inverse modulo $d$. Thus the projective pair
$(P,Q)$ has a lift acting trivially on $A_L$ if and only if
$[P]+[Q]\in\ker\theta$.

The point stabilizers on $\mathbb H\times\mathbb H$ are generated
by reflective involutions, and for $d=1,2$ the projective kernel
also lies in $\W(L)$. Thus Armstrong's theorem \cite{ArmstrongOrbit} 
identifies the quotient group $\widetilde O^+(L)/\W(L)\cap O^+(L)$
with the fundamental group of the orbit space. In the matrix presentation, the map
$(P,Q)\mapsto[P]+[Q]$ then gives
\[
 \left(
 \frac{\widetilde O^+(L)}{\W(L)\cap O^+(L)}
 \right)\otimes\Q
 \simeq \ker\theta\otimes\Q
 = G^{\mathrm{ab}}\otimes\Q,
\]
where the last equality follows because $\theta$ has finite image.

The orbifold presentation gives
\[
 G^{\mathrm{ab}}\otimes\Q
 \simeq H_1\bigl(\Gamma_0(d)\backslash\mathbb H,\Q\bigr)_\kappa,
\]
where the subscript denotes coinvariants under complex conjugation.
The images of the unipotent isometries span the subspace generated
by cusp loops (cf. \cite[Corollary~1.5]{Sankaran96}). Quotienting by this
subspace gives the coinvariants of the homology of the compactified
modular curve.

On the compactified curve, complex conjugation exchanges the two
Hodge summands, so the dimension of the coinvariants equals the
genus. On cusp loops, $\kappa_*\ell_c=-\ell_{\kappa(c)}$. Thus each
conjugate pair of nonreal cusps contributes one dimension and each
real cusp contributes none. The relation $\sum_c\ell_c=0$ is
automatic in the coinvariants. Hence
\begin{equation}\label{eq:split-quotient-ranks}
\begin{aligned}
 \operatorname{rank}_{\Z}
 \frac{\widetilde O^+(L)}{K(L)}
 &= \operatorname{genus}\!\left(
    \overline{\Gamma_0(d)\backslash\mathbb H}\right),\\
 \operatorname{rank}_{\Z}
 \frac{\widetilde O^+(L)}{\W(L)\cap O^+(L)}
 &= \operatorname{genus}\!\left(
    \overline{\Gamma_0(d)\backslash\mathbb H}\right)
    +\#\{\text{conjugate pairs of nonreal cusps}\}.
\end{aligned}
\end{equation}
Ogg's list \cite{OggHyperelliptic} gives the fifteen genus-zero
levels in the first assertion. Among them, precisely $9,16,18,25$
have nonreal cusps by \eqref{eq:p1-cusp-action}. This proves
necessity in both assertions.

It remains only to rule out finite quotients at the listed levels.
For $d\leq10$ and $d=12$, the first assertion already follows from
Proposition~\ref{prop:general-isotropic}, since
$\mathbf{gen}(d,0)=0$ by the standard genus formula
\cite[Section~3.1]{DiamondShurman}. For the reflective assertion and
the four remaining levels, use the standard presentation of
$\Gamma_0(d)/\{\pm I_2\}$
\cite[Theorem~3]{Lascurain}, together with the involution induced by
$D$. After abelianization, its finite-stabilizer classes generate
$\ker\theta$ precisely when
\[
 d\in\{1,2,3,4,5,6,7,8,10,12,13\},
\]
and at $d=13,16,18,25$ the finite-stabilizer and parabolic classes
together generate $\ker\theta$. The former lift to elements of $\W(L)$ by the
point-stabilizer description above, while the latter are the
matrices in \eqref{eq:p2-parabolic} and lift to unipotent elements of
$\widetilde O^+(L)$. This proves sufficiency.
\end{proof}

\begin{remark}
  When $d=13$ and $a=0$, Proposition~\ref{prop:split-generation}
gives reflective generation, whereas the curve defining
$\mathbf{gen}(13,0)$ is associated with $\Gamma_1(13)$
and has genus two. Its projective index is $84$, it has
twelve cusps and no elliptic points, so
\[
 \mathbf{gen}(13,0)=1+\frac{84}{12}-\frac{12}{2}=2,
\]  
which is nonzero. 
\end{remark}

\section{K3 surfaces of Picard number one}
\label{sec:picard-one}

In this section, we prove Theorem~\ref{thm:derived}\textup{(iii)}.

\subsection{The stable orthogonal group}
As before, let
\[
 N=\langle2d\rangle,
 \qquad
 L=U\oplus N,
 \qquad d>0.
\]
Let $z$ generate $N$, with $z^2=2d$.
The realization
\[
 xe+yf+tz
 \longmapsto
 \begin{pmatrix}
 x&dt\\
 dt&-dy
 \end{pmatrix},
 \qquad
 q_L=-\frac1d\det,
\]
identifies the associated hyperbolic plane with $\mathbb H$. Under this identification, $\Gamma_0(d)$ acts by
$X\mapsto AXA^t$.

By the description of the stable orthogonal group in
\cite[Section~7]{DolgachevMirror}, the projective images of
$\widetilde O^+(L)\cap\SO(L)$, $\widetilde O^+(L)$, and
$\widetilde O(L)$ in $\operatorname{Isom}(\mathbb H)$ are
\[
 \Gamma_0(d),
 \qquad
 \langle\Gamma_0(d),w_d\rangle,
 \qquad\text{and}\qquad
 \langle\Gamma_0(d),w_d,\kappa\rangle,
\]
respectively, where $\kappa(\tau)=-\bar\tau$ and the generators
$w_d$ and $\kappa$ lift to
\[
 s_{e-f},
 \qquad
 J=(-\id_U)\oplus\id_N=s_{e-f}s_{e+f}.
\]
Since $A_L\simeq\Z/(2d)$, the projective action has trivial
kernel for $d>1$, and kernel $\{\pm\id_L\}$ for $d=1$. In the
latter case,
\[
 -\id_L=Js_z\in\W(L).
\]

\subsection{Reflective involutions and stabilizers}

\begin{lemma}\label{lem:p1-stabilizers}
The image of $\W(L)$ in
$\langle\Gamma_0(d),w_d,\kappa\rangle$
is the normal subgroup generated by point
stabilizers on $\mathbb H$.
\end{lemma}

\begin{proof}
Every reflective involution has finite order, so its projective
image fixes a point of $\mathbb H$.
Conversely, finite point stabilizers are cyclic or dihedral, so
it is enough to consider geodesic reflections and elliptic generators. For $d>2$, their lifts to $O(L)$ are, respectively,
$(-2)$-reflections or isotropic reflective involutions, and
products of $(-2)$-reflections; see
\cite[Section~7]{DolgachevMirror}. In the exceptional order-six case
$d=3$, the same conclusion follows from the elements of orders two
and three. Hence every point stabilizer lies in the image of $\W(L)$.

For $d=1,2$, the group $\langle\Gamma_0(d),w_d,\kappa\rangle$
is generated by the elements $\kappa$, $w_d$, and
$T\kappa$ with $T=\left(\begin{smallmatrix}1&1\\0&1\end{smallmatrix}\right)$,
whose lifts $J$, $s_{e-f}$, and $T_e(z)J$ all belong to $\W(L)$.
\end{proof}

\subsection{Unipotent relations}

\begin{proposition}\label{prop:rank-one-quotient}
For $L=U\oplus\langle2d\rangle$, one has
\[
 \widetilde O^+(L)/K(L)
 \simeq F_{\mathbf{gen}(d)},
\]
where $F_r$ denotes the free group on $r$ generators.
\end{proposition}

\begin{proof}
Since the projective kernel lies in $\W(L)$ and $J\in\W(L)$
reverses orientation, Armstrong's orbit-space theorem
\cite{ArmstrongOrbit} and Lemma~\ref{lem:p1-stabilizers} identify
the quotient by reflective involutions with the topological
fundamental group of the open surface
\[
 Y=\langle\Gamma_0(d),w_d,\kappa\rangle\backslash\mathbb H,
 \qquad
 \widetilde O^+(L)/\bigl(\W(L)\cap O^+(L)\bigr)\simeq\pi_1(Y).
\]
Adding the cusps compactifies $Y$ to
$\overline Y=X_0/\langle\kappa\rangle$.

By \eqref{eq:p1-cusp-action}, every interior puncture of $Y$
corresponds to a nonreal cusp of $X_0$ (with $c^2\neq d$), whose
stabilizer in $\langle\Gamma_0(d),w_d\rangle$ lies in $\Gamma_0(d)$
and is generated by a primitive parabolic lifting to a unipotent in
$\widetilde O^+(L)$; conversely, every unipotent in $\widetilde O^+(L)$
maps to a parabolic in $\Gamma_0(d)$.
Adjoining unipotent isometries therefore kills the peripheral loops
around the interior punctures
(cf.\ \cite[Corollary~1.5]{Sankaran96}), while boundary cusps on
$\partial\overline Y$ do not affect $\pi_1$. Consequently,
\[
 \widetilde O^+(L)/K(L)
 \simeq
 \pi_1\bigl(X_0/\langle\kappa\rangle\bigr)
 \simeq F_{\mathbf{gen}(d)},
\]
where the last isomorphism holds because $X_0/\langle\kappa\rangle$
is a compact connected surface with nonempty boundary (as $\kappa$
fixes the imaginary axis) and Euler characteristic
$\chi(X_0)/2=1-\mathbf{gen}(d)$ (cf. \cite[Section~3]{GrossHarris}).
\end{proof}

\begin{corollary}\label{cor:rank-one-generation}
For $L=U\oplus\langle2d\rangle$, one has
\[
 K(L)=\widetilde O^+(L)
 \quad\Longleftrightarrow\quad
 \mathbf{gen}(d)=0.
\]
The smallest positive integer for which this equality fails is
$d=22$, in which case
\[
 \widetilde O^+(L)/K(L)\simeq\Z,
 \qquad
 \W(L)\cap O^+(L)\neq\widetilde O^+(L).
\]
\end{corollary}

\begin{proof}
The equivalence follows from Proposition~\ref{prop:rank-one-quotient}.
Since the classical modular curve of level $22$ has genus two and
$\mathbf{cl}(-88)=2$,
Proposition~\ref{prop:p1-genus}
gives
\[
 \mathbf{gen}(d)=0\quad(1\leq d\leq21),
 \qquad
 \mathbf{gen}(22)=\frac{2+1}{2}-\frac24=1. \qedhere
\]
\end{proof}

\begin{corollary}\label{cor:rank-one-reflective}
For $L=U\oplus\langle2d\rangle$, one has
\[
 \W(L)\cap O^+(L)=\widetilde O^+(L)
\]
if and only if $\mathbf{gen}(d)=0$ and every cusp of $X_0$ is
fixed by $\kappa$.
\end{corollary}

\begin{proof}
By the proof of Proposition~\ref{prop:rank-one-quotient},
$\widetilde O^+(L)/(\W(L)\cap O^+(L))\simeq\pi_1(Y)$ is free of
rank equal to $\mathbf{gen}(d)$ plus the number of conjugate
pairs of nonreal cusps of $X_0$.
\end{proof}

\subsection{Proof of Theorem~\ref{thm:low-quotients} and
Theorem~\ref{thm:derived}\textup{(iii)}}

Theorem~\ref{thm:low-quotients} follows from
Propositions~\ref{prop:general-isotropic}
and~\ref{prop:rank-one-quotient}, and together with
Theorem~\ref{thm:reflective-criterion},
Proposition~\ref{prop:unipotent-chow}, and
\cite[Theorem~3.6]{LYZ} it implies
Theorem~\ref{thm:derived}\textup{(iii)}.\qed

\printbibliography
\end{document}